\documentclass[11pt, reqno]{amsart}

\title{Optimal Matching Problem on the Boolean Cube}
\author{Shi Feng}
\address{Cornell University}
\email{\texttt{sf599@cornell.edu}}

\usepackage{geometry}
\usepackage{amsmath, amsfonts, amsthm, amssymb}  
\usepackage{fullpage}

\usepackage[x11names, rgb]{xcolor}
\usepackage{graphicx}
\usepackage{tikz}
\usetikzlibrary{decorations,arrows,shapes}

\usepackage{etoolbox}
\usepackage{enumerate}
\usepackage{listings}
\usepackage{booktabs}
\usepackage{comment}
\usepackage{float}
\usepackage{algorithm2e}
\usepackage[alphabetic]{amsrefs}

\newtheorem{theorem}{Theorem}[section]
\newtheorem{lemma}[theorem]{Lemma}

\newtheorem{corollary}[theorem]{Corollary}

\DeclareMathOperator*{\argmin}{arg\,min}

\begin{document}
\maketitle

\begin{abstract}
We establish upper and lower bounds for the expected Wasserstein distance between the random empirical measure and the uniform measure on the Boolean cube. Our analysis leverages techniques from Fourier analysis, following the framework introduced in \cite{bobkov2021simple}, as well as methods from large deviations theory. 
\end{abstract}

\section{Introduction}
Consider a metric space $(X, d)$ equipped with a probability measure $\mu$. Let $X_1, \dots, X_N$ be i.i.d. random variables taking values in $X$, each distributed according to $\mu$, where $N \in \mathbb{N}$ is large. The (random) empirical measure $\mu_N$ associated with $X_1, \dots, X_N$ is defined as  

\begin{align}
    \mu_N(x) = \frac{1}{N} \sum_{i=1}^N \mathbb{I}(X_i = x),
\end{align}

where $\mathbb{I}$ denotes the indicator function.  

A natural question arises: how different is $\mu_N$ from $\mu$? One of the most widely used metrics to quantify the difference between two probability measures on a metric space is the Wasserstein-1 distance, denoted $\mathsf{W}_1(\mu_N, \mu)$. We will define the Wasserstein-1 distance later in this work. Since $\mu_N$ is itself a random measure, we are particularly interested in its expectation, $\mathbb{E}[\mathsf{W}_1(\mu_N,\mu)]$, and possibly its fluctuations.  

A significant amount of research has been devoted to understanding $\mathbb{E}[\mathsf{W}_1(\mu_N,\mu)]$ in the metric space $X = [0,1]^2$, where $\mu = \text{Unif}([0,1]^2)$ is the uniform distribution. The renowned AKT theorem, proved in \cite{ajtai1984optimal}, identifies the correct leading order asymptotics for $\mathbb{E}[\mathsf{W}_1(\mu_N,\mu)]$ in this setting, providing a key result in the study of Wasserstein distances for empirical measures. Building on these results, Talagrand \cite{talagrand2014upper} used the method of generic chaining to offer a proof that not only confirms the correct asymptotics but also establishes the best possible tail bounds for $\mathsf{W}_1(\mu_N,\mu)$, significantly advancing our understanding of the distribution of the Wasserstein distance in this context. In \cite{ambrosio2019pde}, a different approach was taken, where the limiting constant for a specific metric on $[0,1]^2$ was derived using techniques from partial differential equations (PDEs). This provided a deeper understanding of the geometric structure underlying the Wasserstein distance. More recently, \cite{bobkov2021simple} provided a Fourier analytic proof of the results, notable for its simplicity and clarity. This approach, by simplifying the problem, opens up the possibility that optimal matching problems in other groups or settings could be addressed in a similar way, potentially broadening the applicability of the results beyond the Euclidean case.

Apart from the space $[0,1]^d$, numerous studies have investigated measures on the space $\mathbb{R}^d$. For instance, \cite{fournier2015rate} examines the rate of convergence of $\mathsf{W}_p(\mu_N,\mu)$ for measures $\mu$ in $\mathbb{R}^d$. Furthermore, \cite{dobric1995asymptotics} and \cite{barthe2013combinatorial} establish the existence of a dimension-dependent limiting constant for the rate of convergence under appropriate scaling for arbitrary measures in $\mathbb{R}^d$. In addition, \cite{fournier2023convergence} provides explicit upper bounds, including constants, for $\mathbb{E}[\mathsf{W}_1(\mu_N,\mu)]$ in $\mathbb{R}^d$. Moreover, \cite{del2019central} proves the Central Limit Theorem for $W_p(\mu_N,\mu)$ in $\mathbb{R}^d$. More generally, for measures in arbitrary metric spaces, \cite{weed2019sharp} employs covering numbers to establish convergence rates for the matching problem in such settings.

With the growing popularity of AI and machine learning, there is increasing interest in high-dimensional problems. Consequently, studying high-dimensional matching problems is a worthwhile endeavor. However, only a few papers have explored the optimal matching problem in high-dimensional hypercubes. The first work to investigate high-dimensional matching and discuss the impact of dimension on asymptotics is \cite{talagrand1992matching}. Quantitative results in \cite{bobkov2021simple} also suggest certain bounds in high-dimensional settings; however, these bounds are generally not very precise. For instance, if we seek to determine the number of samples $N$ required to ensure that $\mathbb{E}[\mathsf{W}_1(\mu_N,\mu)] \leq 0.5$ in $[0,1]^n$, Theorem 3 in \cite{bobkov2021simple} states that we need $N \geq e^{\frac{1}{2} n \log(n) + \mathcal{O}(n)}$,while a matching lower bound remains unknown. Furthermore, if $N$ is of polynomial order in $n$, satisfying $\frac{\log(N)}{\log(n)} < \infty$, then very little can be rigorously established about the matching distance in this regime.  

This raises the question of whether there exists a high-dimensional setting where the matching distance can be evaluated with high accuracy for different sample sizes $N$ relative to the dimension $n$. The answer is affirmative: the Boolean cube, $\mathbb{B} = \{-1, 1\}^n$, equipped with the Hamming distance as its metric. Specifically, the Hamming distance between two points $x, y \in \mathbb{B}$ is defined as

\begin{align*}
    d(x, y) = \sum_{i=1}^{n} \mathbb{I}(x_i \neq y_i),
\end{align*}

where $x_i$ denotes the $i$-th coordinate of $x \in \mathbb{B}$. In this work, we consider the uniform probability measure $\mu$ on the Boolean cube, defined as

\begin{align*}
    \mu(A) = \frac{|A|}{2^n},
\end{align*}

for each $A \subseteq \mathbb{B}$. The empirical distribution $\mu_N$, defined in equation (1), is generated by i.i.d. random points $X_1, \dots, X_N \in \mathbb{B}$, and serves as the basis for studying Wasserstein distances between the empirical measure and the true uniform measure $\mu$.

The Boolean cube, $\mathbb{B} = \{-1,1\}^n$, is a fundamental structure in theoretical computer science, extensively studied in areas such as coding theory, computational complexity, and discrete geometry. Its combinatorial structure and well-defined metric properties make it a natural setting for analyzing various high-dimensional optimization problems. Moreover, its discrete nature allows for precise evaluation of the matching distances. This provides a valuable contrast to continuous spaces like $\mathbb{R}^n$, where exact computations are often intractable.

For simplicity of notation, $\mu$ and $\mu_N$ are always the uniform measure and the empirical measure on $\mathbb{B}$ in the following of this paper. And all $\log$ is base $e$. Here we present our main result of this paper.

\begin{theorem}
    When $N = n^{\alpha}$ for $\alpha\geq 2$, we have
    
    \begin{align}
    \label{2}
        \sqrt{\frac{\alpha-2}{4}} \leq \liminf_n \frac{\frac{n}{2}-\mathbb{E}[\mathsf{W}_1(\mu_N,\mu)]}{\sqrt{n\log(n)}} \leq \limsup_n \frac{\frac{n}{2}-\mathbb{E}[\mathsf{W}_1(\mu_N,\mu)]}{\sqrt{n\log(n)}} \leq \sqrt{\frac{\alpha+1}{2}}.
    \end{align}
    
    When $N = e^{\lambda n}$ for $0<\lambda<\log(2)$, define $H(x) = -x\log(x)-(1-x)\log(1-x)$ and take $r^* = H^{-1}(\log(2)-\lambda) \in (0,\frac{1}{2})$. We have
    
    \begin{align}
    \label{3}
        -\frac{1}{\log(\frac{1-r^*}{r^*})} \leq \liminf_n \frac{\mathbb{E}[\mathsf{W}_1(\mu_{N},\mu)] - r^*n}{\log(n)} \leq \limsup_n \frac{\mathbb{E}[\mathsf{W}_1(\mu_{N},\mu)] - r^*n}{\log(n)} \leq \frac{3.5}{\log(\frac{1-r^*}{r^*})}.
    \end{align}
    
    When $N = c2^n$, we have
    
    \begin{align}
    \label{4}
        e^{-c} \leq \liminf_n \mathbb{E}[\mathsf{W}_1(\mu_{N},\mu)] \leq \limsup_n \mathbb{E}[\mathsf{W}_1(\mu_{N},\mu)] \leq \frac{1}{\sqrt{2c}}.
    \end{align}
    
    When $N = c(n)2^n$ for some $c(n) \to \infty$, we have
    
    \begin{align}
    \label{5}
        \sqrt{\frac{1}{2\pi}} \leq \liminf_n \sqrt{c(n)}\mathbb{E}[\mathsf{W}_1(\mu_{N},\mu)] \leq \limsup_n \sqrt{c(n)}\mathbb{E}[\mathsf{W}_1(\mu_{N},\mu)] \leq \frac{\sqrt{2}}{2}.
    \end{align}
    
\end{theorem}

Note that (2) refers to the case of small sample sizes $N$, where $N$ is a polynomial in $n$, the dimension of the space. In (3), we consider the case of intermediate sample sizes, where $N$ is exponentially large but still much smaller than the size of the Boolean cube, $|\mathbb{B}| = 2^n$. In (4), we address large sample sizes, where $N$ is a constant scaling of $|\mathbb{B}|$. Finally, (5) pertains to the scenario of very large sample sizes, where $N$ is much larger than $|\mathbb{B}|$.

Before delving into the proof, we first examine why these bounds are more precise than the existing bounds for the matching problem on $[0,1]^n$. Consider the same question: how many samples $N$ are required to ensure that $\mathbb{E}[\mathsf{W}_1(\mu_{N},\mu)] \leq 0.5$? Equation (4) establishes that  $N = e^{\log(2)n+\Omega(1)}.$  Furthermore, when $2 < \frac{\log(N)}{\log(n)} < \infty$, equation (2) yields  $\mathbb{E}[\mathsf{W}_1(\mu_{N},\mu)] = \frac{n}{2} - \Omega(\sqrt{n \log(n)})$. And when $0<\frac{log(N)}{log(|\mathbb{B}|)}<1$, (3) states that $\mathbb{E}[\mathsf{W}_1(\mu_{N},\mu)] = r^*n+\Omega(log(n))$, where $r^*$ is a function of $\frac{log(N)}{log(|\mathbb{B}|)}$. In conjunction with the concentration bounds presented in Section 6.3, we can provide a precise characterization of $\mathsf{W}_1(\mu_{N},\mu)$ for various values of $N$ relative to $n$. 

Another interesting question concerns the rate of decay of $\mathbb{E}[\mathsf{W}_1(\mu_{N},\mu)]$ as the sample size $N$ tends to infinity for a fixed dimension $n$. Consider the setup where $[0,1]^n$ with $n \geq 3$ is equipped with the uniform measure $m$. According to \cite{dobric1995asymptotics}, there exists a function $f: \mathbb{Z} \to \mathbb{R}_+$ such that  
\begin{align*}  
    \lim_{N\to \infty} \frac{\mathbb{E}[\mathsf{W}_1(m_{N},m)]}{N^{-\frac{1}{n}}} = f(n).  
\end{align*}  
Furthermore, \cite{bobkov2021simple} establishes that $f(n) \leq \mathcal{O}(\sqrt{n})$. However, when considering the Boolean cube, the behavior differs significantly due to its discrete structure. By equation (5), we obtain  
\begin{align*}  
    \limsup_{N\to \infty} \frac{\mathbb{E}[\mathsf{W}_1(\mu_{N},\mu)]}{N^{-\frac{1}{2}}} = \Omega(2^{\frac{n}{2}}).  
\end{align*}  
The $\limsup$ above can be replaced by $\liminf$.

For readers who are interested in results for small $n$, these can be found in Lemma 2.1, Lemma 4.1, Corollary 4.2, and Lemma 5.1. The structure of the paper is as follows. In Section 2, we present the necessary background on the Wasserstein distance, the Fourier transform on the Boolean cube, and the total variation distance. We also explore the connections between these concepts. In Section 3, we prove results (4) and (5) using the Fourier transform. Section 4 is dedicated to proving result (2). In Section 5, we prove result (3) using large deviation methods. In Section 6, we provide a few related theorems with short proofs. Finally, in Section 7, we conclude the paper with a summary of our findings and provide additional remarks.

\section{Preliminary}
\subsection{Wasserstein distance}
Consider a metric space $(X, d)$ and two probability measures $\nu_1$ and $\nu_2$ defined on it. The Wasserstein distance between the two measures is defined as

\begin{align}
    \mathsf{W}_1(\nu_1, \nu_2) = \inf_{\gamma \in \Gamma} \int_{X \times X} d(x_1, x_2) \, d\gamma(x_1, x_2),
\end{align}

where $\Gamma$ represents the set of all couplings between $\nu_1$ and $\nu_2$. This definition can also be rewritten, as established by the famous duality theorem, in an equivalent form:

\begin{align}
    \mathsf{W}_1(\nu_1, \nu_2) = \sup_{f: 1-\text{Lipschitz}} \left( \int_X f(x) \, d\nu_1(x) - \int_X f(x) \, d\nu_2(x) \right),
\end{align}

where the supremum is taken over all 1-Lipschitz functions $f$. We will use the definition in equation (7) to prove results (2), (4), and (5), while definition (6) will be applied to prove result (3). Note that the Wasserstein distance is a metric, meaning that it satisfies the triangle inequality:

\begin{align}
    \mathsf{W}_1(\nu_1, \nu_3) \leq \mathsf{W}_1(\nu_1, \nu_2) + \mathsf{W}_1(\nu_2, \nu_3).
\end{align}

For the duality theorem and the metric property for Wasserstein distance, we refer to Villani's book \cite{villani2009optimal}.

\subsection{Fourier transform}
The following method of Fourier transform for functions on the Boolean cube is well-known, so we will omit proofs for the standard results. For more details, see \cite{o2014analysis}.

For a set $S \subseteq [n]$, define the $S$-th Fourier basis function as

\begin{align*}
    \chi_S = \prod_{i \in S} x_i.
\end{align*}

Then $\{\chi_S \;|\; S \subseteq [n]\}$ forms an orthonormal basis in the function space $\{f: \mathbb{B} \to \mathbb{R}\}$ with respect to the inner product

\begin{align*}
    \langle f, g \rangle = \frac{1}{2^n}\sum_{x \in \mathbb{B}} f(x)g(x).
\end{align*}

Therefore, for any function $f$, we have the Fourier expansion $f = \sum_{S \subseteq [n]} \hat{f}(S) \chi_S$, where $\hat{f}(S) = \langle f, \chi_S \rangle$. The inner product between two functions is also given by

\begin{align*}
    \langle f, g \rangle = \sum_{S \subseteq [n]} \hat{f}(S) \hat{g}(S).
\end{align*}

For a measure $\nu$ on $\mathbb{B}$, define the expected value of $f$ under $\nu$ as

\begin{align*}
    E_{\nu}[f] = 2^n \langle f, \nu \rangle = \sum_{x \in \mathbb{B}} f(x) \nu(x),
\end{align*}

and the expected value with respect to the uniform measure $\mu$ as

\begin{align*}
    E[f] = E_{\mu}[f] = \frac{1}{2^n} \sum_{x \in \mathbb{B}} f(x).
\end{align*}

We intentionally use $E$ instead of $\mathbb{E}$ to denote expectation taken over $\mathbb{B}$, reserving $\mathbb{E}$ for expectation on the probability space that generates the random measure $\mu_N$.
Note that $E[f] = \hat{f}(\emptyset)$, where $\emptyset$ denotes the empty set.

For a function $f$, which can also be interpreted as a measure, we define the $\epsilon$-diffused function as

\begin{align}
    f^\epsilon(x) = \sum_{x' \in \mathbb{B}} (1-\epsilon)^{n - d(x, x')} \epsilon^{d(x, x')} f(x'),
\end{align}

where $d(x, x')$ is the Hamming distance between $x$ and $x'$. The measure $\nu^\epsilon$ can be interpreted as the distribution of a random point $y \in \mathbb{B}$, where $y$ is generated by first picking a point $x$ according to the measure $\nu$, and then flipping each of the $n$ coordinates of $x$ independently with probability $\epsilon$. This provides a probabilistic perspective on how the $\epsilon$-diffusion operator spreads the measure $\nu$ across the Boolean cube, blending nearby points based on their proximity in terms of the Hamming distance.

The functions $\chi_S$ are the orthonormal eigenvectors for the operator of $\epsilon$-diffusion, with the spectrum given by

\begin{align}
    f^\epsilon = \sum_{S \subseteq [n]} (1 - 2\epsilon)^{|S|} \hat{f}(S) \chi_S.
\end{align}

This result is non-trivial and provides a clear connection between the Fourier coefficients of $f$ and the behavior of the $\epsilon$-diffusion operator. For a detailed proof, refer to \cite{o2014analysis}. Since the $\epsilon$-diffusion operator acts linearly and its eigenvectors form an orthonormal basis, it follows that the operator is self-adjoint. The expected value of $f^\epsilon$ with respect to the measure $\nu$ is

\begin{align}
    E_{\nu^\epsilon}[f] = \sum_{S \subseteq [n]} (1 - 2\epsilon)^{|S|} \hat{f}(S) \hat{\nu}(S) = E_{\nu}[f^\epsilon].
\end{align}

If we interpret $\nu^\epsilon$ as the distribution of $y$, generated by sampling $x \sim \nu$ and then flipping each coordinate of $x$ independently with probability $\epsilon$, we can bound the Wasserstein distance between $\nu^\epsilon$ and $\nu$ as:

\begin{align}
    \mathsf{W}_1(\nu^\epsilon, \nu) \leq \mathbb{E}[d(x, y)] = \sum_{i=1}^{n}\mathbb{I}(x_i=y_i) = \epsilon n,
\end{align}

where $d(x, y)$ is the Hamming distance between $x$ and $y$. The first inequality is by coupling $y$ with its generator $x$ and use the first definition (6) of the Wasserstein distance.

Next, we define the $i$-th influence of a function $f: \mathbb{B} \to \mathbb{R}$ as:

\begin{align*}
    \mathsf{Inf}_i(f) = \mathbb{E}\left[\frac{(f(x^{i \to 1}) - f(x^{i \to -1}))^2}{4}\right],
\end{align*}

where $x^{i \to 1}$ denotes the vector $x$ with its $i$-th coordinate set to $1$, and $x^{i \to -1}$ denotes the vector $x$ with its $i$-th coordinate set to $-1$.

If $f$ is 1-Lipschitz, then we have the bound $Inf_i(f) \leq \frac{1}{4}$. The total influence of $f$ is defined as:

\begin{align*}
    \mathsf{Inf}(f) = \sum_{i=1}^n \mathsf{Inf}_i(f).
\end{align*}

Moreover, it holds that:

\begin{align*}
    Inf(f) = \sum_{S \subseteq [n]} |S| \hat{f}(S)^2,
\end{align*}

where $|S|$ represents the size of the subset $S$. We refer to \cite{o2014analysis} for readers who are unfamiliar with the above formula. Therefore, for 1-Lipschitz functions, the Fourier coefficients satisfy:

\begin{align}
    \sum_{S \subseteq [n]} |S| \hat{f}(S)^2 \leq \frac{n}{4}.
\end{align}

\subsection{Connection between 2.1 and 2.2}
Now let's see how the Fourier transform can help us solve the matching distance problem, which seems unrelated at first glance.\\

We first diffuse the random measure $\mu_N$ to the scale we want. By (8) and (12), we have

\begin{align}
    \mathsf{W}_1(\mu_N,\mu) \leq \mathsf{W}_1(\mu_N^\epsilon, \mu) + \mathsf{W}_1(\mu_N^\epsilon, \mu_N) \leq \mathsf{W}_1(\mu_N^\epsilon, \mu) + \epsilon n.
\end{align}

Then we use the second definition (7) of the Wasserstein distance. Note that the integral becomes a sum, since our space $X = \mathbb{B}$ is discrete.

\begin{align*}
    \mathsf{W}_1(\mu_N^\epsilon, \mu) = \sup_{f: 1\text{-Lipschitz}} \sum_{x \in \mathbb{B}} f(x)(\mu_N^\epsilon(x) - \mu(x)) = \sup_{f: 1\text{-Lipschitz}, E[f] = 0} \sum_{x \in \mathbb{B}} f(x)\mu_N^\epsilon(x).
\end{align*}

We are able to restrict our attention to $f$ such that $E[f] = 0$ since constant shifting of $f$ does not affect its Lipschitz condition. In this way, we can remove the $\sum_{x \in \mathbb{B}} f(x)\mu(x)$ term above. After this, we write the sum, or the scaled inner product of $f$ and $\mu_N^\epsilon$, by their Fourier coefficients. By (11),

\begin{align*}
    \sum_{x \in \mathbb{B}} f(x)\mu_N^\epsilon(x) = \sum_{\emptyset \neq S \subseteq [n]} (1 - 2\epsilon)^{|S|} \hat{f}(S)\hat{\mu}_N(S).
\end{align*}

Note that $\hat{f}(S)$ is non-random and fixed by $f$, while $\hat{\mu}_N(S)$ is random and depends on the points $X_1, \ldots, X_N$. Therefore, we aim to separate these two terms and bound them individually. A natural way to achieve this separation is by applying the Cauchy–Schwarz inequality. This gives:

\begin{align*}
    \sum_{\emptyset \neq S \subseteq [n]} (1 - 2\epsilon)^{|S|}\hat{f}(S)\hat{\mu}_N(S) \leq \left(\sum_{\emptyset \neq S \subseteq [n]} |S|\hat{f}(S)^2\right)^{\frac{1}{2}} \left(\sum_{\emptyset \neq S \subseteq [n]} \frac{1}{|S|}(1 - 2\epsilon)^{2|S|}\hat{\mu}_N(S)^2\right)^{\frac{1}{2}}.
\end{align*}

Now, returning to the expected Wasserstein distance, we have:

\begin{align*}
    \mathbb{E}[\mathsf{W}_1(\mu_N^\epsilon, \mu)] \leq \mathbb{E}\left[\sup_{f: 1\text{-Lipschitz}, E[f] = 0} \left(\sum_{\emptyset \neq S \subseteq [n]} |S|\hat{f}(S)^2\right)^{\frac{1}{2}} \left(\sum_{\emptyset \neq S \subseteq [n]} \frac{1}{|S|}(1 - 2\epsilon)^{2|S|}\hat{\mu}_N(S)^2\right)^{\frac{1}{2}}\right].
\end{align*}

By (13) and Jensen's inequality, we can further bound it by

\begin{align*}
    \mathbb{E}[\mathsf{W}_1(\mu_N^\epsilon,\mu)] 
    &\leq \sqrt{\frac{n}{4}}\mathbb{E}\left[\sum_{\emptyset \neq S\in[n]} \frac{1}{|S|}(1-2\epsilon)^{2|S|}\hat{\mu_N}(S)^2\right]^\frac{1}{2} \\
    &= \frac{\sqrt{n}}{2}\left(\sum_{\emptyset \neq S\in[n]} \frac{1}{|S|}(1-2\epsilon)^{2|S|}\mathbb{E}[\hat{\mu_N}(S)^2]\right)^{\frac{1}{2}}.
\end{align*}

If we think more about what is $\hat{\mu_N}(S)$, we realize that

\begin{align*}
    \mathbb{E}[\hat{\mu_N}(S)^2] = \mathbb{E}\left[\left(\frac{1}{N}\sum_{i=1}^{N}\chi_S(X_i)\right)^2\right] = \frac{1}{N^2}\sum_{i=1}^{N}\mathbb{E}[\chi_S(X_i)^2] = \frac{1}{N}.
\end{align*}

The second equality holds due to the i.i.d. property of $X_i$ and the fact that $(\mathbb{E}[\chi_S(X_i)] = 0$ for $ \neq \emptyset$, which eliminates the cross terms. The last equality is by observing $\chi_S(x)^2 = 1$ for all $x\in \mathbb{B}$. Therefore, we obtain the following bound:

\begin{align*}
    \mathbb{E}[\mathsf{W}_1(\mu_N^\epsilon, \mu)] &\leq \frac{\sqrt{n}}{2} \left( \sum_{\emptyset \neq S \in [n]} \frac{1}{|S|}(1-2\epsilon)^{2|S|} \frac{1}{N} \right)^{\frac{1}{2}} \\
    &= \frac{\sqrt{n}}{2\sqrt{N}} \left( \sum_{i=1}^{n} \frac{1}{i} \binom{n}{i} (1-2\epsilon)^{2i} \right)^{\frac{1}{2}}.
\end{align*}

Combining this with equation (14), we obtain the following lemma:

\begin{lemma}
For any \( 0 \leq \epsilon \leq \frac{1}{2} \), we have

\begin{align}
    \mathbb{E}[\mathsf{W}_1(\mu_N,\mu)] &\leq \frac{\sqrt{n}}{2\sqrt{N}} \left( \sum_{i=1}^{n} \frac{1}{i} \binom{n}{i} (1-2\epsilon)^{2i} \right)^{\frac{1}{2}} + \epsilon n.
\end{align}

\end{lemma}

\subsection{Total variation distance}
For two probability measures $ \nu_1, \nu_2 $ on the Boolean cube $ \mathbb{B} $, their total variation distance is defined as

\begin{align*}
    \mathsf{TV}(\nu_1, \nu_2) &= \sum_{x \in \mathbb{B} : \nu_1(x) > \nu_2(x)} \nu_1(x) - \nu_2(x) \\
    &= \sum_{x \in \mathbb{B}} (\nu_1(x) - \nu_2(x))^+,
\end{align*}

where $ (\cdot)^+ $ denotes the positive part of a number (i.e., $ (\nu_1(x) - \nu_2(x))^+ = \max(\nu_1(x) - \nu_2(x), 0) $).

The key observation is the following relation between the total variation distance and the Wasserstein distance:

\begin{align}
    \mathsf{TV}(\nu_1, \nu_2) &\leq \mathsf{W}_1(\nu_1, \nu_2) \leq n \cdot \mathsf{TV}(\nu_1, \nu_2),
\end{align}

which states that the total variation distance is a lower bound and $ n $ times the total variation distance is an upper bound for the Wasserstein distance.

This result can be explained as follows: The total variation distance represents how much mass needs to be shifted from $ \nu_1 $ to $ \nu_2 $. The Wasserstein distance $ W $ takes into account the "cost" of moving mass. The cost of moving mass is at least 1 and at most $ n $ (diameter of $\mathbb{B}$). Therefore, the Wasserstein distance will always lie between the total variation distance and $ n $ times the total variation distance. For detailed proof, refer to \cite{villani2009optimal}.

\section{For (very) large size $N$}
Now we focus on the case where $N = c(n)2^n$ for some function $c(n)$ depending on $n$.

\begin{lemma}
    Suppose $N = c(n)2^n$, then
    
    \begin{align}
        \limsup_n \sqrt{c(n)}\mathbb{E}[\mathsf{W}_1(\mu_N,\mu)] \leq \frac{\sqrt{2}}{2}.
    \end{align}
    
\end{lemma}
\begin{proof}
    We borrow Lemma 2.1 with $\epsilon = 0$.
    
    \begin{align*}
        \mathbb{E}[\mathsf{W}_1(\mu_N,\mu)] \leq \frac{\sqrt{n}}{2\sqrt{c(n)2^n}}\left(\sum_{i=1}^{n}\frac{1}{i}{n\choose i}\right)^\frac{1}{2}.
    \end{align*}
    
    Then we observe the following:
    
    \begin{align*}
        \lim_n\frac{n}{2^n}\sum_{i=1}^{n}\frac{1}{i}{n\choose i}  
        &= \lim_n \frac{n}{2^n} \sum_{i\in[\frac{n}{2}-n^{2/3},\frac{n}{2}+n^{2/3}]} \frac{1}{i}{n\choose i}  + \lim_n \frac{n}{2^n} \sum_{i\notin[\frac{n}{2}-n^{2/3},\frac{n}{2}+n^{2/3}]} \frac{1}{i}{n\choose i}\\
        &= \lim_n \frac{n}{2^n}\cdot 2^n(1-o(n^{-1}))\left(\frac{2}{n}+o(n^{-1})\right) + \lim_n \frac{n}{2^n}\cdot 2^no(n^{-1})O(1)\\
        &= 2+0=2.
    \end{align*}
    
    The second equality follows from the observation that $\frac{1}{2^n}\sum_{i \notin [\frac{n}{2} - n^{2/3}, \frac{n}{2} + n^{2/3}]} \binom{n}{i} = o(n^{-1})$, which can be established using standard concentration inequalities. For a more concrete proof, we can utilize equations (23) and (26) in later sections. Here, $B(x, r)$ is defined above equation (23). Specifically, we have  
    \begin{align*}
        \frac{1}{2^n} \sum_{i \notin [\frac{n}{2} - n^{2/3}, \frac{n}{2} + n^{2/3}]} \binom{n}{i}  
        &\leq \frac{2 |B(x, \frac{n}{2} - n^{2/3})|}{2^n}, \\
        &= e^{n(-\log(2) + H(\frac{1}{2} - n^{-1/3})) + \log(2)}, \\
        &\leq e^{n(-\log(2) + \log(2) - 2n^{-2/3}) + \log(2)}, \\
        &= e^{-2n^{1/3} + \log(2)} = o(n^{-1}).
    \end{align*}

    Therefore, we have
    
    \begin{align*}
        \limsup_n \sqrt{c(n)}\mathbb{E}[\mathsf{W}_1(\mu_N,\mu)] \leq \limsup_n \frac{1}{2}\left(\frac{n}{2^n}\sum_{i=1}^{n}\frac{2}{n}{n\choose i} \right)^{\frac{1}{2}} = \frac{\sqrt{2}}{2}.
    \end{align*}
    
\end{proof}
Therefore, we have successfully proven the upper bound for (4) and (5) by taking $c(n)$ or $c(n) = c$. Now, all we need to prove is the lower bound.

\begin{proof}[proof of (4) and (5)]
    The upper bounds in (4) and (5) follow directly from Lemma 3.1. For the lower bounds, we use the total variation distance. Suppose $N = cn$. Then by (16),
    
    \begin{align*}
        \liminf_n \mathbb{E}[\mathsf{W}_1(\mu_N,\mu)] 
        &\geq \liminf_n \mathbb{E}[\mathsf{TV}(\mu_N, \mu)] \\
        &= \liminf_n \sum_{x \in \mathbb{B}} \mathbb{E}[(\mu(x) - \mu_N(x))^+] \\
        &\geq \liminf_n \sum_{x \in \mathbb{B}} \frac{1}{2^n} \mathbb{P}(\mu_N(x) = 0) \\
        &= e^{-c}.
    \end{align*}
    
    The second equality follows from the linearity of expectation. The final equality comes from the observation that
    
    \begin{align*}
        \lim_n \mathbb{P}(\mu_{c 2^n}(x) = 0) = \lim_n \left(1 - \frac{1}{2^n}\right)^{c 2^n} = e^{-c}.
    \end{align*}

    Now, suppose $N = c(n) 2^n$ and $c(n) \to \infty$. Then, by the Central Limit Theorem, for any $x \in \mathbb{B}$,
    
\begin{align*}
    2^n \sqrt{c(n)} (\mu_N(x) - \mu(x)) \to N(0,1).
\end{align*}

Note that by the linearity of expectation,

\begin{align*}
    \lim_n \sqrt{c(n)} \mathbb{E}[\mathsf{TV}(\mu_N, \mu)] &= \lim_n \sqrt{c(n)} \sum_{x \in \mathbb{B}} \mathbb{E}[(\mu_N(x) - \mu(x))^+].
\end{align*}

By the symmetry of $\mu$ and $\mathbb{B}$, we have

\begin{align*}
    \sum_{x \in \mathbb{B}} \mathbb{E}[(\mu_N(x) - \mu(x))^+] = 2^n \mathbb{E}[(\mu_N(y) - \mu(y))^+]
\end{align*}

for any fixed $y \in \mathbb{B}$. Thus,

\begin{align*}
    \lim_n \sqrt{c(n)} \mathbb{E}[\mathsf{TV}(\mu_N, \mu)] &= \lim_n \mathbb{E}\left[\sqrt{c(n)} 2^n (\mu_N(y) - \mu(y))^+\right] \\
    &= \mathbb{E}[N(0,1)^+] \\
    &= \sqrt{\frac{1}{2\pi}}.
\end{align*}

Together with (16), we obtain the lower bound in (5).

\end{proof}

\section{For small size $N$}
Now, we focus on the case when $N = n^{\alpha}$. In this case, we need to consider the $\epsilon$-diffusion in Lemma 2.1. The question is: what value of $\epsilon$ should we choose? 

We begin by presenting a well-known bound for the binomial coefficient $\binom{n}{i}$, as referenced in \cite{ash2012information}:  

\begin{align}
    \binom{n}{i} \leq e^{H\left(\frac{i}{n}\right)n},
\end{align}  

where $H(x) = -x\log(x) - (1-x)\log(1-x)$ is the entropy function for a $Bernoulli(x)$ random variable. This function $H$ has already been introduced in (3), and it will be consistently applied throughout the remainder of this paper.  

Now, we proceed to state the main lemma of this section.

\begin{lemma}
    For any $N\geq n^2$, we have
    
    \begin{align}
        \mathbb{E}[\mathsf{W}_1(\mu_N,\mu)] \leq \frac{1-\sqrt{e^{\frac{\log(N)-2\log(n)}{n}}-1}}{2}\cdot n + 1.
    \end{align}
    
\end{lemma}
\begin{proof}
Note that in the summation of Lemma 2.1, the term $\frac{1}{i}$ is polynomial, while both ${n \choose i}$ and $(1 - 2\epsilon)^{2i}$ are exponential. Therefore, we can discard the "trivial" $\frac{1}{i}$ term and focus solely on the exponential terms.

\begin{align*}
    \mathbb{E}[\mathsf{W}_1(\mu_N,\mu)] 
    &\leq \frac{\sqrt{n}}{\sqrt{N}}\left(\sum_{i=1}^{n}{n\choose i}(1-2\epsilon)^{2i}\right)^{\frac{1}{2}} + \epsilon n\\
    &\leq \frac{\sqrt{n}}{\sqrt{N}}\left(n\sup_{i\in [n]}\left\{{n\choose i}(1-2\epsilon)^{2i}\right\}\right)^{\frac{1}{2}} + \epsilon n\\
    &\leq\exp\left\{\log(n)-\frac{1}{2}\log(N)+\frac{1}{2}\sup_{i\in [n]}\left\{nH(\frac{i}{n})+2i\log(1-2\epsilon) \right\}\right\} + \epsilon n\\
    &\leq \exp\left\{\log(n)-\frac{1}{2}\log(N)+\frac{n}{2}\sup_{0\leq k\leq 1}\left\{H(k)+2k\log(1-2\epsilon) \right\}\right\} + \epsilon n. 
\end{align*}

The last inequality follows by taking $k = \frac{i}{n}$ and maximizing over a larger continuous set. Let $g(k) = H(k) + 2k\log(1 - 2\epsilon) = -k\log(k) - (1 - k)\log(1 - k) + 2k\log(1 - 2\epsilon)$. We find its maximum by taking derivatives.

\begin{align*}
    \frac{d}{dk} g(k) = -\log(k) + \log(1-k)+2\log(1-2\epsilon) = \log\left(\frac{(1-k)(1-2\epsilon)^2}{k}\right) = 0,
\end{align*}

which implies

\begin{align*}
    (1-k)(1-2\epsilon)^2 = k \implies k = \frac{(1-2\epsilon)^2}{1+(1-2\epsilon)^2}.
\end{align*}

Next, we take the second derivative. For $0 < k < 1$, we have

\begin{align*}
    \frac{d^2}{dk^2} g(k) = -\frac{1}{k}-\frac{1}{1-k} < 0.
\end{align*}

Therefore, the unique maximizer for $g(k)$ is $k = \frac{(1-2\epsilon)^2}{1+(1-2\epsilon)^2}$. In this case,

\begin{align*}
    g(k) 
    &= -\frac{(1-2\epsilon)^2}{1+(1-2\epsilon)^2}\log\left(\frac{(1-2\epsilon)^2}{1+(1-2\epsilon)^2}\right) - \frac{1}{1+(1-2\epsilon)^2} \log\left(\frac{1}{1+(1-2\epsilon)^2}\right) + 2\frac{(1-2\epsilon)^2}{1+(1-2\epsilon)^2}\log(1-2\epsilon)\\
    &= \log(1+(1-2\epsilon)^2)-\frac{(1-2\epsilon)^2}{1+(1-2\epsilon)^2}\log((1-2\epsilon)^2) + 2\frac{(1-2\epsilon)^2}{1+(1-2\epsilon)^2}\log(1-2\epsilon)\\
    &= \log(1+(1-2\epsilon)^2).
\end{align*}

We then bring it back to bound the Wasserstein distance.

\begin{align*}
    \mathbb{E}[\mathsf{W}_1(\mu_N,\mu)]
    &\leq \exp\left\{\log(n)-\frac{1}{2}\log(N)+\frac{n}{2}\sup_{0\leq k\leq 1}\left\{g(k) \right\}\right\} + \epsilon n\\
    &= \exp\left\{\log(n)-\frac{1}{2}\log(N)+\frac{n}{2}\log(1+(1-2\epsilon)^2)\right\} + \epsilon n.
\end{align*}

Now, we want to choose an $\epsilon$ such that the term in the exponent is smaller than 0. In this case, we have $\mathbb{E}[\mathsf{W}_1(\mu_N,\mu)] \leq 1 + \epsilon n$. To achieve this, we need

\begin{align*}
    \log(n)-\frac{1}{2}\log(N)+\frac{n}{2}\log(1+(1-2\epsilon)^2) \leq 0
    &\iff \log(1+(1-2\epsilon)^2) \leq \frac{2}{n}(\frac{1}{2}\log(N)-\log(n))\\
    &\iff  (1-2\epsilon)^2 \leq e^{\frac{\log(N)-2\log(n)}{n}}-1\\
    &\iff \epsilon \geq \frac{1-\sqrt{e^{\frac{\log(N)-2\log(n)}{n}}-1}}{2}.
\end{align*}

Therefore, by picking such an $\epsilon$, we obtain

\begin{align*}
    \mathbb{E}[\mathsf{W}_1(\mu_N,\mu)] \leq e^{0} + \frac{1-\sqrt{e^{\frac{\log(N)-2\log(n)}{n}}-1}}{2}\cdot n = \frac{1-\sqrt{e^{\frac{\log(N)-2\log(n)}{n}}-1}}{2}\cdot n + 1
\end{align*}

\end{proof}
Now, let's explore how Lemma 4.1 can be applied to small values of $N$, specifically when $N = n^{\alpha}$ for $\alpha \geq 2$. We begin by applying the simple inequality $e^x \geq 1 + x$. Using this, along with Lemma 4.1, we obtain the following bound:

\begin{align}
    \mathbb{E}[\mathsf{W}_1(\mu_N,\mu)] &\leq \frac{1 - \sqrt{\frac{\log(N) - 2\log(n)}{n}}}{2} \cdot n + 1 =\frac{n}{2} - \frac{\sqrt{\log(N) - 2\log(n)}}{2} \cdot \sqrt{n} + 1.
\end{align}

From this bound, we can deduce a slightly stronger result than what is stated in equation (2) of the main theorem, particularly when $\alpha = 2$. We have decided to exclude this case from the main theorem and present the result as a corollary instead.

\begin{corollary}
    Suppose $N = cn^2$ for $c>1$, then
    
\begin{align}
    \mathbb{E}[\mathsf{W}_1(\mu_{N},\mu)] \leq \frac{n}{2}-\frac{\sqrt{\log(c)}}{2}\sqrt{n}+1.
\end{align}

    Suppose $N = n^{\alpha}$ for $\alpha \geq 2$, then
    
\begin{align}
    \mathbb{E}[\mathsf{W}_1(\mu_{N},\mu)] \leq \frac{n}{2}-\frac{\sqrt{(\alpha-2)\log(n)}}{2}\sqrt{n}+1.
\end{align}

\end{corollary}
Corollary 4.2 is proved by substituting the value of $N$ into equation (20). Now, we can proceed with the proof of inequality (2) in the main theorem.

\begin{proof}[proof of (2)]
The upper bound (the left-most inequality in (2)) follows directly from Corollary 4.2, so we will focus only on the lower bound. First, we define the ball $B(x, r)$ on $\mathbb{B}$ for $x \in \mathbb{B}$ and $r \in \mathbb{N}$ as

\begin{align*}
    B(x, r) &= \{y \in \mathbb{B} \mid d(x, y) \leq r\}.
\end{align*}

Note that $|B(x, r)| = \sum_{i=0}^{r} {n \choose i}$. By well-known inequalities (see \cite{ash2012information}, page 121), for $r \leq \frac{n}{2}$, we have the following bound.

\begin{align}
    e^{H(\frac{r}{n})n-\frac{1}{2}\log(2n)} \leq |B(x,r)| \leq e^{H(\frac{r}{n})n}.
\end{align}

Now, let’s return to the lower bound. Consider the set

\begin{align*}
    \bigcup_{i=1}^N B(X_i, r).
\end{align*}

If the size of this set is small compared to $2^n$, then most points on $\mathbb{B}$ must be moved a distance larger than $r + 1$ to find an $X_i$. In this case, the expected Wasserstein distance would be greater than $(r + 1) \cdot \mathbb{P}(x \notin \bigcup_{i=1}^N B(X_i, r))$. If $\mathbb{P}(x \notin \bigcup_{i=1}^N B(X_i, r)) \geq \frac{n-1}{n}$, then

\begin{align*}
    \mathbb{E}[\mathsf{W}_1(\mu_N,\mu)] \geq (r + 1) \cdot \frac{n-1}{n} \geq r
\end{align*}

since $r \leq \frac{n}{2}$. To achieve the goal, it is sufficient to find an $r$ that satisfies the following inequality (*):

\begin{align}
    \frac{|\bigcup_{i=1}^N B(X_i, r)|}{2^n} &\leq \frac{N |B(X_1, r)|}{2^n} \leq \frac{N e^{H\left( \frac{r}{n} \right) n}}{2^n} \stackrel{*}{\leq} \frac{1}{n}.
\end{align}

In this case, $N = n^\alpha$, so we need to find the largest $r \leq \frac{n}{2}$ such that

\begin{align}
    H\left(\frac{r}{n}\right) &\leq \frac{1}{n} \log \left( \frac{2^n}{n^{\alpha+1}} \right) = \log(2) - \frac{(\alpha+1)\log(n)}{n}.
\end{align}

We claim that for all $0 \leq x \leq 1$,

\begin{align}
    H(x) &\leq \log(2) - 2 \left( x - \frac{1}{2} \right)^2.
\end{align}

One way to prove this is by considering the function $h(x) = H(x) + 2 \left( x - \frac{1}{2} \right)^2 = -x \log(x) - (1-x) \log(1-x) + 2 \left( x - \frac{1}{2} \right)^2$. We compute the first derivative:

\begin{align*}
    \frac{d}{dx} h(x) &= -\log(x) + \log(1 - x) + 4 \left( x - \frac{1}{2} \right) = 0.
\end{align*}

One solution to this equation is $x = \frac{1}{2}$. To confirm that this is the maximum, we compute the second derivative:

\begin{align*}
    \frac{d^2}{dx^2} h(x) &= -\frac{1}{x} - \frac{1}{1 - x} + 4 \leq 0
\end{align*}

for $0 < x < 1$. This shows that $x = \frac{1}{2}$ is the unique maximizer. In this case, $h\left(\frac{1}{2}\right) = \log(2)$, proving inequality (26). Returning to inequality (25), and taking $r = \frac{n}{2} - \sqrt{\frac{(\alpha+1)\log(n)}{2n}}$, we get

\begin{align*}
    H\left(\frac{r}{n}\right) &= H\left( \frac{1}{2} - \sqrt{\frac{(\alpha+1)\log(n)}{2n}} \right) \leq \log(2) - \frac{(\alpha+1)\log(n)}{n}.
\end{align*}

Therefore, we have

\begin{align*}
    \mathbb{E}[\mathsf{W}_1(\mu_N,\mu)] \geq r =\frac{n}{2} - \sqrt{\frac{\alpha+1}{2}n\log(n)},
\end{align*}

which implies the lower bound in (2).
\end{proof}

\section{For intermediate size $N$}

For intermediate values of $N$, one might suggest using Lemma 4.1. However, the result provided by Lemma 4.1 is not "accurate" for $N = e^{\lambda n}$, where $0 < \lambda < \log(2)$. Interested readers can try this themselves. Since the Fourier transform method does not help in this case, we entirely abandon it and instead use a different method, namely, the large deviation method.

We first define a different diffusion process than the one introduced in Section 2. For a function $f: \mathbb{B} \to \mathbb{R}$, which can also be interpreted as a measure, we define

\begin{align}
    P_r f(x) &= \frac{1}{|B(x, r)|} \sum_{y \in B(x, r)} f(y)
\end{align}

for $r \leq \frac{n}{2}$, where $B(x, r)$ is the ball defined in Section 4 with the size bound given in (23). To see the difference between this diffusion process and the $\epsilon$-diffusion introduced in Section 2, one can interpret $P_r \nu$ as the distribution of $y'$, where $y'$ is generated by first picking a random $x \sim \nu$, then choosing a uniform random point in $B(x, r)$. Since the diffusion never passes beyond the ball with radius $r$, we have for any probability measure $\nu$,

\begin{align*}
    \mathsf{W}_1(P_r \nu, \nu) &\leq r
\end{align*}

by coupling $y'$ with its generator $x$. Furthermore, by (8), we have

\begin{align*}
    \mathsf{W}_1(\mu_N,\mu) &\leq \mathsf{W}_1(P_r \mu_N, \mu) + \mathsf{W}_1(P_r \mu_N, \mu_N) \leq \mathsf{W}_1(P_r \mu_N, \mu) + r.
\end{align*}

We then use the total variation bound (16), the linearity of expectation, and symmetry to obtain

\begin{align*}
    \mathbb{E}[\mathsf{W}_1(\mu_N,\mu)] 
    &\leq \mathbb{E}[\mathsf{W}_1(P_r \mu_N, \mu)] + r \\
    &\leq n \mathbb{E}[\mathsf{TV}(P_r \mu_N, \mu)] + r \\
    &= \left( n \sum_{x \in \mathbb{B}} \mathbb{E} \left[ \left( P_r \mu_N(x) - \frac{1}{2^n} \right)^+ \right] \right) + r \\
    &= \left( n 2^n \mathbb{E} \left[ \left( P_r \mu_N(y) - \frac{1}{2^n} \right)^+ \right] \right) + r,
\end{align*}

where $y$ is any fixed point in $\mathbb{B}$. Note that for any $\delta > 0$, we have

\begin{align*}
    \mathbb{E} \left[ \left( P_r \mu_N(y) - \frac{1}{2^n} \right)^+ \right] &\leq \mathbb{P} \left( P_r \mu_N(y) \geq (1 + \delta) \frac{1}{2^n} \right) + \frac{\delta}{2^n}.
\end{align*}

Therefore, we have

\begin{align}
    \mathbb{E}[\mathsf{W}_1(\mu_N,\mu)] \leq n2^n\mathbb{P}\left(P_r\mu_N(y) \geq (1+\delta)\frac{1}{2^n}\right) + \delta n + r
\end{align}

To bound the probability above, we have the following lemma.

\begin{lemma}
    Denote $p = \frac{|B(y,r)|}{2^n}$, we have
    
    \begin{align}
        \mathbb{P}\left(P_r\mu_N(y)> (1+\delta)\frac{1}{2^n}\right) \leq e^{N(\log(1+p\delta) - \log(1+\delta)(p+p\delta))},
    \end{align}
    
    and therefore
    
    \begin{align}
        \mathbb{E}[\mathsf{W}_1(\mu_N,\mu)] \leq n2^ne^{N(\log(1+p\delta) - \log(1+\delta)(p+p\delta))} + \delta n + r
    \end{align}
    
    for any $\delta>0$ and $r\leq \frac{n}{2}$.
\end{lemma}
\begin{proof}
    Note that $\mu_N$ is generated by i.i.d. uniform points $X_1, \dots, X_N$. Denote $\mathbb{I}_i(y, r)$ as the indicator function that $X_i \in B(y, r)$. Then, we have $\mathbb{E}[\mathbb{I}_i(y, r)] = 2^{-n} |B(y, r)| = p$, and

\begin{align*}
    P_r \mu_N(y) &= \frac{1}{N | B(y, r) |} \sum_{i=1}^{N} \mathbb{I}_i(y, r).
\end{align*}

Using the exponential Markov inequality, for any $\lambda \geq 0$, we have

\begin{align*}
    \mathbb{P} \left( P_r \mu_N(y) > (1 + \delta) \frac{1}{2^n} \right) 
    &= \mathbb{P} \left( \sum_{i=1}^{N} \mathbb{I}_i(y, r) \geq N p (1 + \delta) \right) \\
    &\leq \inf_{\lambda} \mathbb{E} \left[ e^{\lambda \mathbb{I}_1(y, r)} \right]^N e^{-\lambda N p (1 + \delta)} \\
    &= \inf_{\lambda} (1 - p + p e^{\lambda})^N e^{-\lambda N p (1 + \delta)} \\
    &= \inf_{\lambda} e^{N \left( \log(1 - p + p e^{\lambda}) - \lambda p (1 + \delta) \right)}.
\end{align*}

    Now we need to decide which value of $\lambda$ to pick. A simple approach is to take the derivative of the exponent with respect to $\lambda$ and then set the derivative to 0.

Denote $l(\lambda) = \log(1 - p + p e^{\lambda}) - \lambda p (1 + \delta)$, then

\begin{align*}
    \frac{d}{d\lambda} l(\lambda) = 0 
    &\iff \frac{p e^{\lambda}}{1 - p + p e^{\lambda}} = p (1 + \delta) \\
    &\iff e^{\lambda} = \frac{(1 + \delta)(1 - p)}{1 - (1 + \delta) p}.
\end{align*}

If we look at (28) again, we want $\delta = O\left(\frac{1}{n}\right)$ and $p = O(e^{-cn})$ for some constant $c$, which implies $p \ll \delta$. Therefore, by setting $p = 0$ in the above equality, we can approximate $\lambda = \log(1 + \delta)$ as being close to the minimizer. Note that we don't need to pick the exact minimizer. We can simply substitute this $\lambda$ into the formula, yielding

\begin{align*}
    \mathbb{P} \left( P_r \mu_N(y) > \frac{1}{2^n}(1 + \delta) \right) 
    &\leq \inf_{\lambda} e^{N \left( \log(1 - p + p e^{\lambda}) - \lambda p (1 + \delta) \right)} \\
    &\leq e^{N \left( \log(1 + p \delta) - \log(1 + \delta)(p + p \delta) \right)}.
\end{align*}

And (30) follows directly from (28) and (29).

\end{proof} 
Let's focus on (30). Ideally, we want the first two terms in the inequality to be less than 1, in which case we can bound the distance by $r+2$. Therefore, we first choose $\delta = \frac{1}{n}$ for the second term. Next, we aim to find an $r$ for the first term such that

\begin{align}
    R(r) = n2^n e^{N \left( \log\left(1 + \frac{p}{n}\right) - \log\left(1 + \frac{1}{n}\right)\left(p + \frac{p}{n}\right) \right)} \leq 1.
\end{align}

With this setup, we now begin proving (3) in our main theorem.

\begin{proof}[proof of (3)]
For simplicity of notation, $\limsup_n f(n) \leq \limsup_n g(n)$ means there exists $N \in \mathbb{Z}$ such that $f(n) \leq g(n)$ for all $n \geq N$ (a similar definition applies for $\liminf_n$). 

We first prove the upper bound following the above argument. Take $r = H^{-1}(\log(2) - \lambda)n + C\log(n)$ for some $C > 0$. When we write $H^{-1}(*)$, we always require $0 \leq * \leq \log(2)$ so that its value is uniquely determined. Note that $H'(x) = -\log(x) + \log(1-x) \in (0, \infty)$ for $0 < x < \frac{1}{2}$. Then, by (23),

\begin{align*}
    p &= \frac{|B(y,r)|}{2^n} \\
    &\geq \exp\left\{-\frac{\log(2n)}{2}+n(H(\frac{r}{n}) - \log(2))\right\}\\
    &= \exp\left\{-\frac{\log(2n)}{2} + n\left(H\left(\frac{H^{-1}(\log(2)-\lambda)n}{n}+\frac{C\log(n)}{n}\right) - \log(2)\right)\right\}\\
    &= \exp\left\{-\frac{\log(2n)}{2} + n\left(-\lambda + H'(H^{-1}(\log(2)-\lambda))\cdot\frac{C\log(n)}{n} + O\left(\frac{\log(n)^2}{n^2}\right) \right) \right\}.
\end{align*}

Denote $r^* = H^{-1}(\log(2) - \lambda)$, which is the same as the $r^*$ in (3). Then, $r = r^*n + C\log(n)$, and for any $\epsilon > 0$,  
\begin{align}
    \liminf_n p \geq \liminf_n \exp\left\{-\frac{\log(2n)}{2} - \lambda n + (C - \epsilon)H'(r^*)\log(n)\right\}.
\end{align}

Recall $R(r)$ in (31). Note that Taylor's theorem tells us that $\log(1 + x) = x - \frac{x^2}{2} + O(x^3)$. Therefore, for any $\epsilon > 0$,

\begin{align*}
    \limsup_nR(r) 
    &= \limsup_n n2^n\exp\left\{N\left(\frac{p}{n}+O(\frac{p^2}{n^2}) - \left(\frac{1}{n}-\frac{1}{2n^2}+O(\frac{1}{n^3})\right)(p+\frac{p}{n})\right)\right\}\\
    &= \limsup_n \exp\left\{\log(n)+n\log(2) - N\left(\frac{p}{2n^2}+O(\frac{p}{n^3}+\frac{p^2}{n^2})\right) \right\}\\
    &\leq \limsup_n \exp\left\{\log(n)+n\log(2) - N\left(\frac{(1-\epsilon)p}{2n^2}\right) \right\}.
\end{align*}

The last inequality follows by observing that $p = o(1)$ when $\lambda > 0$. Now, we bring the lower bound of $p$ from (32) and $N = e^{\lambda n}$ into the above inequality to obtain, for any $\epsilon > 0$,

\begin{align*}
    \limsup_nR(r) 
    &\leq \limsup_n \exp\left\{\log(n)+n\log(2) - e^{\lambda n}\left(\frac{(1-\epsilon)\exp\{-\frac{\log(2n)}{2}-\lambda n + (C-\epsilon)H'(r^*)\log(n)\}}{2n^2}\right) \right\}\\
    &\leq \limsup_n \exp\left\{\log(n)+n\log(2) - \exp\left\{\frac{5\log(n)}{2} + (C-\epsilon)H'(r^*)\log(n)\right\} \right\}.
\end{align*}

In the last inequality, we place everything in the exponent and remove the constants by taking $C - 2\epsilon$, which is equivalent to $C - \epsilon$ by appropriately choosing another $\epsilon$. Now, if we take $C = \frac{3.5 + \epsilon}{H'(r^*)}$, we have

\begin{align*}
    \limsup_nR(r) \leq \limsup_n \exp\{\log(n)+n\log(2) - n^{1+\epsilon}\} = 0.
\end{align*}

Thus, by taking $r = r^*n + \frac{(3.5 + \epsilon)\log(n)}{H'(r^*)} = r^*n + \frac{(3.5 + \epsilon)\log(n)}{\log\left(\frac{1 - r^*}{r^*}\right)}$ and $\delta = \frac{1}{n}$ in (30), we obtain, for any $\epsilon > 0$,

\begin{align*}
    \limsup_n \mathbb{E}[\mathsf{W}_1(\mu_N,\mu)] \leq \limsup_n (0+1+r) = \limsup_n 1+ r^*n+ \frac{3.5+\epsilon}{\log(\frac{1-r^*}{r^*})}\log(n).
\end{align*}

And this implies the upper bound in (3). Now, we proceed to prove the lower bound for (3). The idea is very similar to the method used to prove the lower bound for (2) in the previous section. Recall equation (24); it suffices to find an $r$ such that

\begin{align*}
    \frac{Ne^{H(\frac{r}{n})n}}{2^n} \leq \frac{1}{n}.
\end{align*}

We take $r = r^*n - \frac{1 + \epsilon}{H'(r^*)} \log(n)$ for some $\epsilon > 0$ and $r^* = H^{-1}(\log(2) - \lambda)$. By substituting $N = e^{-\lambda}$, we obtain

\begin{align*}
    \limsup_n \frac{N e^{H\left(\frac{r}{n}\right) n}}{2^n} &= \limsup_n \exp \left\{ \lambda n + H\left(r^* - \frac{1 + \epsilon}{H'(r^*)} \cdot \frac{\log(n)}{n}\right) n - \log(2) n \right\} \\
    &= \limsup_n \exp \left\{ n[\lambda + H(r^*) - \log(2)] - (1 + \epsilon)\log(n) + O\left(\frac{\log(n)^2}{n}\right) \right\} \\
    &\leq \limsup_n \frac{1}{n}.
\end{align*}

The last inequality follows from $\lambda + H(r^*) - \log(2) = 0$. Therefore, by choosing such $r$, for any $\epsilon > 0$,

\begin{align*}
    \liminf_n \mathsf{W}_1(\mu_N,\mu) \geq \liminf_n r^*n - \frac{1+\epsilon}{H'(r^*)} = \liminf_n r^*n - \frac{1+\epsilon}{\log(\frac{1-r^*}{r^*})},
\end{align*}

which implies the lower bound in (3).
\end{proof}

\section{General properties of the matching problem}
As promised in the introduction, we will prove a few related theorems that can be applied to matching problems in other spaces. Section 6.1 concerns the problem of finding a measure $\nu$ that minimizes the expected matching distance to the empirical distribution $m_N$, that is, $\mathbb{E}[\mathsf{W}_1(m_N, \nu)]$. Section 6.2 establishes that the expected matching distance is a monotone decreasing function with respect to the number of samples $N$. Section 6.3 provides variance and concentration inequalities for the matching distance.

\subsection{A best guess for the empirical distribution}
Consider a probability measure $m$ and its empirical measure $m_N$ on the space $(X, d)$. A natural question is whether, among all possible probability measures on $X$, the measure $m$ is the "best guess" for approximating $m_N$.

Let $\mathcal{V}$ denote the set of all probability measures on $X$. To formally define what we mean by "best guess," we introduce  
\begin{align*}  
    \mathcal{M}_N = \argmin_{\nu \in \mathcal{V}} \mathbb{E}[\mathsf{W}_1(m_N, \nu)],  
\end{align*}  
which represents the set of optimal measures minimizing the expected Wasserstein distance from $m_N$. The question then reduces to whether $m \in \mathcal{M}_N$.

This definition can be compared to the results in \cite{dereich2013constructive}, where the authors analyze $\mathbb{E}[\mathsf{W}_1(m_N, m)]$ in relation to $\inf_{\nu}\mathsf{W}_1(\nu, m)$, where the infimum is taken over the set of probability measures supported on $N$ points. In contrast, in our definition of $\mathcal{M}_N$, rather than approximating $m$, we aim to approximate $m_N$ while optimizing over all possible measures. Since $m_N$ is random, the expectation in the definition of $\mathcal{M}_N$ is essential.

The significance of $\mathcal{M}_N$ can be illustrated through the following scenario. Suppose that $N$ individuals will appear in the space $X$ at random locations, with each location being independently drawn from the distribution $m$. Before their arrival, we aim to distribute food across $X$ in such a way that each individual receives an equal share while minimizing the distance they must travel to collect it. The optimal strategy is likely to select a measure $\nu$ from $\mathcal{M}_N$ and distribute the food according to $\nu$.

We provide a sufficient condition for $m \in \mathcal{M}_N$, which holds in many interesting cases, including the Boolean cube.

\begin{theorem}
    Consider a compact topological group $G$ equipped with a metric $d$ that induces its topology. Suppose that $d$ is invariant under the group action, meaning that for any elements $h_1, h_2, g \in G$, we have  
    \begin{align*}  
        d(h_1, h_2) = d(gh_1, gh_2).  
    \end{align*}  
    
    Now, assume that $m$ is a group-invariant measure (Haar measure) on $G$. That is, for any subset $H \subseteq G$ and any $g \in G$,  
    \begin{align*}  
        m(H) = m(gH).  
    \end{align*}  
    
    Then, for all $N \geq 1$, 
    \begin{equation}
        m \in\mathcal{M}_N = \argmin_{\nu \in \mathcal{V}} \mathbb{E}[\mathsf{W}_1(m_N, \nu)],
    \end{equation}  
    meaning that $m$ minimizes the expected Wasserstein distance to the empirical measure $m_N$.

\end{theorem}
\begin{proof}
    It suffices to prove that for any $\nu \in \mathcal{V}$, we have the inequality $\mathbb{E}[\mathsf{W}_1(m_N, m)] \leq \mathbb{E}[\mathsf{W}_1(m_N, \nu)]$. To this end, we first observe that the measure $m$ has full support: for any open set $O$, it holds that $m(O) > 0$. This follows from the fact that the topology on $G$ is generated by open balls of the form $B(g,r) = \{h \in G \mid d(g,h) < r\}$. Suppose there exists an open ball $O = B(g,r)$ such that $m(O) = 0$. In that case, we can find a finite union of balls $\cup_{i=1}^n B(g_i,r)$ that covers the entire group $G$, as $G$ is compact. Specifically, since $G = \bigcup_{i=1}^n B(g_i,r)$, we have
    \begin{align*}
        m(G) \leq m\left(\bigcup_{i=1}^n B(g_i,r)\right) \leq \sum_{i=1}^{n} m(B(g_i,r)) = n \cdot m(B(g,r)) = 0,
    \end{align*}
    which leads to a contradiction. The third equality follows from the invariance of $m$ under translations, i.e., $m(B(g_i,r)) = m(gg_i^{-1}B(g_i,r)) = m(B(g,r))$. Thus, we conclude that $m$ has full support.
    
    For any distribution $\nu$, there exists a sequence of measures $\nu_1, \nu_2, \dots$ such that each $\nu_i$ is absolutely continuous with respect to $m$, and we have $\lim_n \mathsf{W}_1(\nu, \nu_n) = 0$. To construct this sequence explicitly, for any $H \subseteq G$, define
    $\nu_n(H) = \int_H f_n(g) \, dm(g),$
    where $f_n$ is the Radon–Nikodym derivative, given by
    \begin{align*}
        f_n(g) = \frac{\nu(B(g, \frac{1}{n}))}{m(B(g, \frac{1}{n}))} < \infty,
    \end{align*}
    since $m$ has full support. Here, we can interpret $\nu_n$ as the distribution of a random variable $y_n$, where $y_n$ is generated by first sampling a point $x \sim \nu$, and then sampling $y_n$ from the conditional distribution of $m$ in the ball $B(x, \frac{1}{n})$. By coupling $y_n$ and $x$, it follows that
    $\mathsf{W}_1(\nu, \nu_n) \leq \frac{1}{n},$
    and hence, $\lim_n \mathsf{W}_1(\nu, \nu_n) = 0$.
    
    Therefore, by the triangle inequality for the Wasserstein metric (8), it suffices to prove that for any $\nu \in \mathcal{V}$ that is absolutely continuous with respect to $m$, we have $\mathbb{E}[\mathsf{W}_1(m_N, m)] \leq \mathbb{E}[\mathsf{W}_1(m_N, \nu)]$. Specifically, in this case for any $\nu \in \mathcal{V}$ and $\epsilon > 0$, there exists a measure $\nu_\epsilon$ that is absolutely continuous with respect to $m$ and satisfies $\mathsf{W}_1(\nu, \nu_\epsilon) < \epsilon$. Then, we have
    \begin{align*}
        \mathbb{E}[\mathsf{W}_1(m_N, m)] \leq \mathbb{E}[\mathsf{W}_1(m_N, \nu_\epsilon)] \leq \mathbb{E}[\mathsf{W}_1(m_N, \nu)] + \mathbb{E}[\mathsf{W}_1(\nu_\epsilon, \nu)] = \mathbb{E}[\mathsf{W}_1(m_N, \nu)] + \epsilon.
    \end{align*}
    
    By the Radon–Nikodym theorem, there exists a function $f: G \to \mathbb{R}_+$ such that for each $H \subseteq G$, we can express
    \begin{align*}
        \nu(H) = \int_H f(g) \, dm(g).
    \end{align*}
    Next, define an operator $T_g: \mathcal{V} \to \mathcal{V}$ for each $g \in G$, such that for any subset $H \subseteq G$,
    \begin{align*}
        T_g \nu(H) = \nu(gH).
    \end{align*}
    We now claim that for any $g \in G$,
    \begin{align*}
        \mathbb{E}[\mathsf{W}_1(m_N, T_g \nu)] = \mathbb{E}[\mathsf{W}_1(T_g m_N, T_g \nu)] = \mathbb{E}[\mathsf{W}_1(m_N, \nu)].
    \end{align*}
    The first equality follows from the group invariance of $m$, which ensures that $m_N$ and $T_g m_N$ have the same distribution on $\mathcal{V}$. The second equality follows from the group invariance of the distance function $d$. Specifically, if $\gamma(h, h')$ represents the optimal coupling for $\mathsf{W}_1(m_N, \nu)$, then we can define the coupling $\gamma'(h, h') = \gamma(gh, gh')$ for $\mathsf{W}_1(T_g m_N, T_g \nu)$, which achieves the same matching distance. This argument holds in both directions, completing the proof of the equality.

    Note that for any $H \subseteq G$,
    \begin{align*}
        \int_G T_g \nu(H) \, dm(g) &= \int_G \nu(gH) \, dm(g) \\
        &= \int_G \int_H f(gh) \, dm(h) \, dm(g) \\
        &= \int_H \int_G f(gh) \, dm(g) \, dm(h) \\
        &= \int_H \int_G f(g) \, dm(g) \, dm(h) \\
        &= \int_H 1 \, dm(h) \\
        &= m(H)
    \end{align*}
    We can exchange the integrals because all quantities are positive. Therefore, we can write $m = \int_G T_g \nu \, dm(g)$. Finally, by the convexity of optimal cost (\cite{villani2009optimal}, Theorem 4.8), we have
    \begin{align*}
        \mathbb{E}[\mathsf{W}_1(m_N, m)] &= \mathbb{E}\left[W\left(m_N, \int_G T_g \nu \, dg\right)\right] \\
        &\leq \mathbb{E}\left[\int_G \mathsf{W}_1(m_N, T_g \nu) \, dg\right] \\
        &= \mathbb{E}\left[\int_G \mathsf{W}_1(m_N, \nu) \, dg\right] \\
        &= \mathbb{E}\left[\mathsf{W}_1(m_N, \nu)\right]
    \end{align*} 
\end{proof}

Now, let's apply this theorem to some interesting spaces.

\begin{corollary}
    (a) Consider the space $X = \mathbb{B}^n$, the metric $d$ as the Hamming distance, and the measure $m$ as the uniform measure.

    (b) Consider the space $X = \mathbb{T}^n \cong [0,1]^n/\{(x,0) = (x,1), (0,y) = (1,y)\}$, the metric $d_p(x,y) = ||x - y||_p$ for some $p \geq 1$, and the measure $m$ as the uniform measure.

    (c) Consider the space $X = S^{n-1}$ as the unit sphere, the metric $d(x,y) = \arccos(\langle x, y \rangle)$ as the surface distance, and the measure $m$ as the uniform measure.

    In all three cases above,
    \begin{align}
        m \in \mathcal{M}_N
    \end{align}
    for all $N \geq 1$.
\end{corollary}

\begin{proof}
    \textbf{(a)} Consider $\mathbb{B}^n$ as a group with the product defined by $x \cdot y = (x_1 y_1, \dots, x_n y_n)$, where $x = (x_1, \dots, x_n)$ and $y = (y_1, \dots, y_n)$ are elements of $\mathbb{B}^n$.

    \textbf{(b)} Consider $\mathbb{T}^n$ as the Lie group $\bigoplus_{i=1}^n \mathbb{R}/\mathbb{Z}$, the direct sum of $n$ copies of the circle group, where each $\mathbb{R}/\mathbb{Z}$ represents the unit circle in the real line, often identified with the quotient of the real numbers modulo 1.
    
    \textbf{(c)} Consider $S^{n-1}$ as the Lie group $SO(n)/SO(n-1)$, where $SO(n)$ denotes the special orthogonal group of degree $n$, representing rotations in $n$-dimensional Euclidean space, and $SO(n-1)$ represents rotations in the $(n-1)$-dimensional subspace orthogonal to a chosen direction.
\end{proof}

\subsection{Monotonicity of the expected matching distance}

In the introduction, we posed questions such as: How many samples, $N$, are required to ensure that $\mathbb{E}[\mathsf{W}_1(\mu_N, \mu)] \leq 0.5$? However, could it be possible that $\mathbb{E}[\mathsf{W}_1(\mu_N, \mu)] \leq \mathbb{E}[\mathsf{W}_1(\mu_{N+1}, \mu)]$? If this were the case, the question would lose its significance.  

In this section, we resolve this concern by proving that the expected matching distance is a decreasing function of the sample size $N$.

\begin{theorem}
    Consider any metric space $(X, d)$ and two (possibly different) probability measures $\nu$ and $\nu'$. Let $\nu_N$ be the empirical measure for $\nu$. Then, for any $N \geq 1$, there exists a probability space such that
    
    \begin{align}
        \mathbb{E}[\mathsf{W}_1(\nu_N, \nu') \mid \mathsf{W}_1(\nu_{N+1}, \nu')] \geq \mathsf{W}_1(\nu_{N+1}, \nu').
    \end{align}
    
    Therefore, we have the inequality
    
    \begin{align}
        \mathbb{E}[\mathsf{W}_1(\nu_{N+1}, \nu')] \leq \mathbb{E}[\mathsf{W}_1(\nu_{N}, \nu')].
    \end{align}

\end{theorem}
\begin{proof}
    Given $\mathsf{W}_1(\nu_{N+1}, \nu')$, suppose $\nu_{N+1}$ is generated by random points $X_1, \dots, X_{N+1} \in X$. Note that the conditional distribution of $X_1, \dots, X_{N+1}$ might not be independent, but they are still exchangeable. Recall the second definition of Wasserstein distance in (7). For any $\epsilon > 0$, suppose $f_\epsilon$ is 1-Lipschitz and

    \begin{align*}
        \int_X f_\epsilon \, d\nu_{N+1} - \int_X f_\epsilon \, d\nu' \geq \mathsf{W}_1(\nu_{N+1}, \nu') - \epsilon.
    \end{align*}

    Note that the existence of $f_\epsilon$ is guaranteed by (7), but it is not fixed and depends on $\nu_{N+1}$. We then generate $\nu_N$ by keeping $\{X_1, \dots, X_{N}\}$. By the conditional exchangeability of $X_1, \dots, X_{N+1}$, we have
    
    \begin{align*}
        \mathbb{E}\left[\int_X f_\epsilon d\nu_{N} \;\Bigg{|}\; \mathsf{W}_1(\nu_{N+1},\nu')\right] = \int_X f_\epsilon d\nu_{N+1}.
    \end{align*}
    
    Therefore,
    
    \begin{align*}
        \mathbb{E}[\mathsf{W}_1(\nu_N,\nu')\;|\;\mathsf{W}_1(\nu_{N+1},\nu')] 
        &= \mathbb{E}\left[\sup_{f:1-Lipschitz} \int_X f d\nu_{N} - \int_Xf d\nu' \;\Bigg{|}\; \mathsf{W}_1(\nu_{N+1},\nu')\right]\\
        &\geq \mathbb{E}\left[\int_X f_\epsilon d\nu_{N} - \int_X f_\epsilon d\nu' \;\Bigg{|}\; \mathsf{W}_1(\nu_{N+1},\nu')\right]\\
        &= \mathbb{E}\left[\int_X f_\epsilon d\nu_{N+1} - \int_X f_\epsilon d\nu' \;\Bigg{|}\; \mathsf{W}_1(\nu_{N+1},\nu')\right]\\
        &\geq \mathsf{W}_1(\nu_{N+1},\nu') - \epsilon.
    \end{align*}
    
    Since $\epsilon$ is arbitrary, we have proved (35). Then the tower property implies (36).
\end{proof}

\subsection{Variance and concentration inequalities}
\begin{theorem}
Consider any metric space $(X, d)$ and a probability measure $\nu$ on it such that for two independent random variables $X_1, X_2 \sim \nu$, we have $\mathbb{E}[d(X_1, X_2)^2] < \infty$. Then, for any (possibly different) probability measure $\nu'$ such that $\mathbb{E}[\mathsf{W}_1(\nu_1, \nu')] < \infty$, we have

\begin{align}
    \mathsf{Var}(\mathsf{W}_1(\nu_N,\nu')) \leq \frac{\mathbb{E}[d(X_1,X_2)^2]}{2N}.
\end{align}

\end{theorem}
\begin{proof}
    By the monotonicity in Theorem 6.3, we have $\mathbb{E}[\mathsf{W}_1(\nu_N, \nu')] < \infty$ for all $N \geq 1$. Suppose $\nu_N$ is generated by $X_1, \dots, X_N$, and let $X_1', \dots, X_N'$ be i.i.d. copies of it. Denote $X(i) = (X_1', \dots, X_i', X_{i+1}, \dots, X_N)$. Let $\nu_N(i)$ be the probability measure generated by $X(i)$. Then, we can write $\mathsf{W}_1(\nu_N, \nu') = \mathsf{W}_1(\mu_N(0), \nu')$, and by the Efron–Stein inequality \cite{efron1981jackknife}, we have

    \begin{align*}
        \mathsf{Var}(\mathsf{W}_1(\nu_N,\nu')) \leq \frac{1}{2}\sum_{i=1}^{N}\mathbb{E}[(\mathsf{W}_1(\nu_n(i),\nu') - \mathsf{W}_1(\nu_n(i-1),\nu'))^2].
    \end{align*}
    
    Note that
    
    \begin{align*}
        |\mathsf{W}_1(\nu_N(i),\nu') - \mathsf{W}_1(\nu_N(i-1),\nu')| \leq \mathsf{W}_1(\nu_N(i),\nu_N(i-1)) \leq \frac{d(X_i,X_i')}{N}.
    \end{align*} 
    
    The first inequality follows from the triangle inequality of the Wasserstein metric (8). The second inequality is obtained by directly moving the mass from $X_i$ to $X_i'$. Therefore,
    
    \begin{align*}
        \mathsf{Var}(\mathsf{W}_1(\nu_N,\nu')) \leq \frac{1}{2}\sum_{i=1}^{N}\mathbb{E}\left[\left(\frac{d(X_i,X_i')}{N}\right)^2\right] = \frac{\mathbb{E}[d(X_1,X_2)^2]}{2N}
    \end{align*}
\end{proof}

\begin{theorem}
    Consider a bounded metric space $(X, d)$ with diameter $D = \sup_{x_1, x_2 \in X} d(x_1, x_2)$ and a probability measure $\nu$ on $X$ with the empirical measure $\nu_N$. Then, for any (possibly different) probability measure $\nu'$ and for all $t \geq 0$, the following concentration inequalities hold:  

\begin{align}
    P(\mathsf{W}_1(\nu_N, \nu') - M(\mathsf{W}_1(\nu_N, \nu')) \geq t) \leq 2e^{-\frac{t^2 N}{4D^2}},
\end{align}  
and  
\begin{align}
    P(\mathsf{W}_1(\nu_N, \nu') - M(\mathsf{W}_1(\nu_N, \nu')) \leq -t) \leq 2e^{-\frac{t^2 N}{4D^2}},
\end{align}  
where $M(\mathsf{W}_1(\nu_N, \nu'))$ denotes the median of $\mathsf{W}_1(\nu_N, \nu')$.

\end{theorem}
\begin{proof}
    Since $X$ is bounded, it follows that $M(\mathsf{W}_1(\nu_N, \nu')) \leq D < \infty$ for all $N \geq 1$. The result can be established using Talagrand's inequality \cite{talagrand1995concentration}.  

    Consider the product space $X^N$ and define the subsets  
    $$
    A = \{ x \in X^N \mid \mathsf{W}_1(\delta_x, \nu') \leq M(\mathsf{W}_1(\nu_N, \nu')) \}
    $$  
    and  
    $$
    B = \{ x \in X^N \mid \mathsf{W}_1(\delta_x, \nu') \geq M(\mathsf{W}_1(\nu_N, \nu')) \},
    $$ 
    where $\delta_x = \sum_{i=1}^N \frac{1}{N}\delta_{x_i}$ as a probability measure in $X$. Since $\nu_N$ is a product measure on $X^N$ and $M(\mathsf{W}_1(\nu_N,\nu'))$ is the median, we have  
    $$
    \nu_N(A) \geq \frac{1}{2}, \quad \text{and} \quad \nu_N(B) \geq \frac{1}{2}.
    $$  
    
    For a subset $A \subseteq X^N$ and an element $x \in X^N$, define the function $g(A, x)$ as follows:  
    \begin{align*}
        g(A, x) = \sup_{|\beta| \leq 1} \inf_{y \in A} \sum_{i=1}^N \beta_i I(x_i \neq y_i),
    \end{align*}  
    where $\beta \in \mathbb{R}^N$ is a weight vector satisfying $|\beta| \leq 1$. Since $\nu_N$ is a product measure on $X^N$, Talagrand's inequality implies that for any $u \geq 0$, we have  
\begin{align}
    \mu_N(\{x \in X^N \mid g(A, x) \geq u\}) \mu_N(A) \leq e^{-\frac{u^2}{4}}.
\end{align}  

Now, consider the set  
$$
A' = \{ x \in X^N \mid \mathsf{W}_1(\delta_x, \nu') \geq M(\mathsf{W}_1(\nu_N, \nu')) + t \}.
$$  
We claim that for any $a = (a_1, \dots, a_N) \in A$ and $a' = (a_1', \dots, a_N') \in A'$, the following inequality holds:  
$$
\sum_{i=1}^{N} I(a_i \neq a_i') \geq \frac{tN}{D}.
$$  
This follows directly from the triangle inequality for the Wasserstein metric:  
\begin{align*}
    t &\leq \mathsf{W}_1(\delta_{a'}, \nu') - \mathsf{W}_1(\delta_{a}, \nu') \\
    &\leq \mathsf{W}_1(\delta_{a}, \delta_{a'}) \\
    &\leq \frac{D \sum_{i=1}^{N} I(a_i \neq a_i')}{N}.
\end{align*}  

By choosing $\beta_i = \frac{1}{\sqrt{N}}$ for all $i$, we obtain, for any $a' \in A'$,  
\begin{align*}
    g(A, a') &\geq \inf_{y \in A} \sum_{i=1}^{N} \sqrt{\frac{1}{N}} I(a_i \neq a_i')  \\
    &\geq \frac{t \sqrt{N}}{D}.
\end{align*}  

Thus, we conclude that  
$$
A' \subseteq \{ x \in X^N \mid g(A, x) \geq \frac{t\sqrt{N}}{D} \}.
$$  
Applying inequality (40), we obtain  
\begin{align*}
    P(\mathsf{W}_1(\nu_N, \nu') - M(\mathsf{W}_1(\nu_N, \nu')) \geq t) &= \mu_N(A')\\
    &\leq \mu_N(\{ x \in X^N \mid g(A, x) \geq \frac{t\sqrt{N}}{D} \})\\
    &\leq \frac{e^{-\frac{t^2 N}{4D^2}}}{\mu_N(A)}\\
    &\leq 2e^{-\frac{t^2 N}{4D^2}}.
\end{align*}  

A similar argument applies to (39) by considering the set  
$$
B' = \{ x \in X^N \mid \mathsf{W}_1(\delta_x, \nu') \leq M(\mathsf{W}_1(\nu_N, \nu')) - t \}.
$$  

\end{proof}

\section{Conclusion}
In this paper, we establish bounds on $\mathbb{E}[\mathsf{W}_1(\mu_N,\mu)]$ for $\mathbb{B} = \{-1,1\}^n$, considering different sample sizes $N$ in relation to the dimension $n$. For the upper bound, we employ Fourier analysis and large deviation techniques, while for the lower bound, we utilize total variation distance and volume estimation.  

An important observation is that the empirical process is "almost optimal" for covering the space in high dimensions, as the lower bound obtained via volume estimation closely matches the upper bound. This suggests that in Euclidean space $\mathbb{R}^n$, the expected matching distance is also very close to the lower bound determined by volume constraints in high dimensions. A precise estimation of this behavior remains an open question.

\section*{Acknowledgement}
I'm grateful to my advisor Philippe Sosoe for his support on this paper. I'm also thankful to Kengo Kato and Soumik Pal for suggesting related references.

\bibliographystyle{alpha}
\bibliography{References}

\begin{bibdiv}
\begin{biblist}

\bib{ajtai1984optimal}{article}{
      author={Ajtai, Mikl{\'o}s},
      author={Koml{\'o}s, J{\'a}nos},
      author={Tusn{\'a}dy, G{\'a}bor},
       title={On optimal matchings},
        date={1984},
     journal={Combinatorica},
      volume={4},
       pages={259\ndash 264},
}

\bib{ash2012information}{book}{
      author={Ash, Robert~B},
       title={Information theory},
   publisher={Courier Corporation},
        date={2012},
}

\bib{ambrosio2019pde}{article}{
      author={Ambrosio, Luigi},
      author={Stra, Federico},
      author={Trevisan, Dario},
       title={A pde approach to a 2-dimensional matching problem},
        date={2019},
     journal={Probability Theory and Related Fields},
      volume={173},
       pages={433\ndash 477},
}

\bib{barthe2013combinatorial}{incollection}{
      author={Barthe, Franck},
      author={Bordenave, Charles},
       title={Combinatorial optimization over two random point sets},
        date={2013},
   booktitle={S{\'e}minaire de probabilit{\'e}s xlv},
   publisher={Springer},
       pages={483\ndash 535},
}

\bib{bobkov2021simple}{article}{
      author={Bobkov, Sergey~G},
      author={Ledoux, Michel},
       title={A simple fourier analytic proof of the akt optimal matching theorem},
        date={2021},
     journal={The Annals of Applied Probability},
      volume={31},
      number={6},
       pages={2567\ndash 2584},
}

\bib{del2019central}{article}{
      author={Del~Barrio, Eustasio},
      author={Loubes, Jean-Michel},
       title={Central limit theorems for empirical transportation cost in general dimension},
        date={2019},
}

\bib{dereich2013constructive}{inproceedings}{
      author={Dereich, Steffen},
      author={Scheutzow, Michael},
      author={Schottstedt, Reik},
       title={Constructive quantization: Approximation by empirical measures},
        date={2013},
   booktitle={Annales de l'ihp probabilit{\'e}s et statistiques},
      volume={49},
       pages={1183\ndash 1203},
}

\bib{dobric1995asymptotics}{article}{
      author={Dobri{\'c}, Vladimir},
      author={Yukich, Joseph~E},
       title={Asymptotics for transportation cost in high dimensions},
        date={1995},
     journal={Journal of Theoretical Probability},
      volume={8},
       pages={97\ndash 118},
}

\bib{efron1981jackknife}{article}{
      author={Efron, Bradley},
      author={Stein, Charles},
       title={The jackknife estimate of variance},
        date={1981},
     journal={The Annals of Statistics},
       pages={586\ndash 596},
}

\bib{fournier2015rate}{article}{
      author={Fournier, Nicolas},
      author={Guillin, Arnaud},
       title={On the rate of convergence in wasserstein distance of the empirical measure},
        date={2015},
     journal={Probability theory and related fields},
      volume={162},
      number={3},
       pages={707\ndash 738},
}

\bib{fournier2023convergence}{article}{
      author={Fournier, Nicolas},
       title={Convergence of the empirical measure in expected wasserstein distance: non-asymptotic explicit bounds in \(\mathbb{R}^d\)},
        date={2023},
     journal={ESAIM: Probability and Statistics},
      volume={27},
       pages={749\ndash 775},
}

\bib{o2014analysis}{book}{
      author={O'Donnell, Ryan},
       title={Analysis of boolean functions},
   publisher={Cambridge University Press},
        date={2014},
}

\bib{talagrand2014upper}{book}{
      author={Talagrand, Michel},
       title={Upper and lower bounds for stochastic processes},
   publisher={Springer},
        date={2014},
      volume={60},
}

\bib{talagrand1992matching}{article}{
      author={Talagrand, Michel},
       title={Matching random samples in many dimensions},
        date={1992},
     journal={The Annals of Applied Probability},
       pages={846\ndash 856},
}

\bib{talagrand1995concentration}{article}{
      author={Talagrand, Michel},
       title={Concentration of measure and isoperimetric inequalities in product spaces},
        date={1995},
     journal={Publications Math{\'e}matiques de l'Institut des Hautes Etudes Scientifiques},
      volume={81},
       pages={73\ndash 205},
}

\bib{villani2009optimal}{book}{
      author={Villani, C{\'e}dric},
      author={others},
       title={Optimal transport: old and new},
   publisher={Springer},
        date={2009},
      volume={338},
}

\bib{weed2019sharp}{article}{
      author={Weed, Jonathan},
      author={Bach, Francis},
       title={Sharp asymptotic and finite-sample rates of convergence of empirical measures in wasserstein distance},
        date={2019},
     journal={Bernoulli},
      volume={25},
      number={4A},
       pages={2620\ndash 2648},
}

\end{biblist}
\end{bibdiv}
\end{document}